\documentclass[11pt]{article}
\usepackage{geometry}                
\usepackage{graphicx}
\usepackage{amssymb}
\usepackage{color}
\usepackage[all,cmtip]{xy}
\usepackage{tikz}
\usepackage{amsthm}
\usepackage{amsmath}
\usepackage{cleveref}
\usepackage{enumerate} 
\usepackage{multirow}
\usepackage{multicol}
\usepackage{amsfonts}
\usepackage{stmaryrd}
\usepackage{MnSymbol}
\newtheorem{theorem}{Theorem}

\newtheorem{thm}{Theorem}[section]
\newtheorem{prop}[thm]{Proposition}
\newtheorem{lem}[thm]{Lemma}
\newtheorem{cor}[thm]{Corollary}

\theoremstyle{definition}
\newtheorem{defn}[thm]{Definition}
\newtheorem{rem}[thm]{Remark}
\newtheorem{example}[thm]{Example}

\DeclareMathOperator{\Homeo}{Homeo_+}

\DeclareMathOperator{\S1}{\mathbb{S}^1}
\DeclareMathOperator{\SL}{SL}

\DeclareMathOperator{\Z}{\mathbb{Z}}
\DeclareMathOperator{\D}{\mathbb{D}}
\DeclareMathOperator{\E}{E}

\theoremstyle{plain}

\newcommand\blfootnote[1]{
 \begingroup
    \renewcommand\thefootnote{}\footnote{#1}
    \addtocounter{footnote}{-1}
\endgroup}

\title{Torsion at the Threshold for Mapping Class Groups}
\date{}
\author{Solomon Jekel and Rita Jim\'enez Rolland}

\begin{document}

\maketitle

\begin{abstract}
The  mapping class group ${\Gamma}_g^ 1$ of a closed orientable surface of genus $g \geq 1$  with one marked point  can be identified, by the Nielsen action, with a subgroup of the group  of orientation preserving homeomorphims of the circle. This inclusion pulls back the powers of the  discrete universal Euler class 
producing classes $\E^n \in H^{2n}({\Gamma}_g^ 1;\mathbb{Z})$ for all  $n\geq 1$.  In this paper we study the power $n=g,$ and prove:  $\E^g$  is a torsion class which generates a cyclic subgroup  of $H^{2g}({\Gamma_g^ 1; \mathbb Z})$ whose order is a positive integer multiple of  $4g(2g+1)(2g-1).$
\end{abstract}


\section{Introduction}

The  mapping class group ${\Gamma}_g^ 1$ of a closed orientable surface $\Sigma_g$ of genus $g \geq 1$  with one marked point  can be identified, by the Nielsen action, with a subgroup of the group  $\Homeo\S1$ of orientation preserving homeomorphims of the circle.
The inclusion $\rho:\Gamma_g^1\hookrightarrow \Homeo\S1$ pulls back the powers of the {\it discrete universal Euler class} 
producing classes $\E^n \in H^{2n}({\Gamma}_g^ 1;\mathbb{Z})$ for all  $n\geq 1$.  At the ``threshold" $n=g$ we prove the following. 
\begin{theorem}\label{MAIN} 
The  cohomology class $\E^g$  generates a finite cyclic subgroup  of $H^{2g}({\Gamma_g^ 1; \mathbb Z})$ whose order is a positive integer multiple of $4g(2g+1)(2g-1).$
\end{theorem}

All powers of the Euler class $\E$ are non-trivial,  effectively because ${\Gamma}_g^ 1$ has non-trivial finite cyclic subgroups; see for instance \cite{JJR20}. So a basic component of the proof of Theorem A  is distinguishing finite cyclic from infinite cyclic, or in the language of the Universal Coefficient Theorem, determining whether $\E^g$ lies in the $Ext$-term or in the $Hom$-term.
Associated to an action of the group ${\Gamma}_g^ 1$ there is a bi-simplicial set, 
a double chain complex and a total chain complex which compute the homology and cohomology of the group. 
In the orbit complex of the bi-simplicial set we construct a $2g$-chain ``dual'' to a representative of the class $\E^g$, and we attempt to lift it to a 
$2g$-cycle on the total complex. If it lifts the  class $\E^g$ is in the $Hom$-term.  We show however that there is an obstruction to the lifting,  hence the class $\E^g$ is in the $Ext$-term.

A special feature of the lifting process is that it is technically difficult to construct and analyze the obstruction in the double chain complex determined by the Nielsen action, but manageable for a modified action which we call the ``inversive action". 

For the inversive action we consider ${\Gamma}_g^ 1$ as acting on  homotopy classes of  un-oriented based loops on a closed surface of genus $g$ by identifying an element of the fundamental group of the surface with its inverse. Then, even though ${\Gamma}_g^ 1$ no longer acts directly on the fundamental group of $\Sigma_g$, the formal bi-simplicial constructions still determine the homology and cohomology of $\Gamma_g^1$ and there is a cohomology class which pulls back to the class $\E$. It is in this context that we  study the behavior of the class $\E^g.$

The paper is organized as follows. Section \ref{Pre} deals with preliminaries on bi-simplicial sets and homology. In Section \ref{MCG} we discuss the Nielsen action and the inversive action of the mapping
class group $\Gamma_g^1$, and Section \ref{Euler} presents the discrete Euler class in the context of the constructions in the previous two sections. 
Lifting using the inversive action is carried out in Section \ref{SecTranCycle}. It follows closely the analysis in \cite{Je12} using the so called ``projective action". It is assumed there that the action can be defined on the circle, but the inversive action, which is the correct formulation, cannot be constructed in that way.
An advantage of the inversive action is that the obstruction to the lifting, which we call the {\it transition cycle}, lies within a single summand of the total complex associated to the action where a  combinatorial analysis leads to a formula for the cycle. 
In Section \ref{Torsion} we show that $\E^g$ is a non-trivial torsion class in $H^{2g}(\Gamma_g^1;\mathbb{Z})$; see Theorem \ref{Theo:TorsionClass}  and Remark \ref{NONTRIV}. This  can be deduced indirectly from the fact that all the powers $\E^n$ of the Euler class are non-trivial in integer cohomology, \cite{JJR20}, but are known to be trivial in rational cohomology when $n\geq g$ by algebro-geometric techniques \cite{Io02, Lo95}. Our intrinsic proof (Sections \ref{DETECT} and \ref{MAINProof}) enables us to detect torsion of order $2g-1.$ As opposed to torsion of order $4g$ and $2g+1$, it is not  constructed using periodic elements in the  mapping class group.  

\section{Preliminaries}\label{Pre}

\subsection{Homeomorphisms of the circle and the universal Euler class}\label{EULER}

Consider the group  $\Homeo\S1$ of orientation preserving homeomorphisms of the circle $\S1$ with the discrete topology.  Let $\Homeo(\S1)_\tau$  denote the same group with the compact-open topology. 
The starting point for our analysis of the discrete Euler class is the theorem of  Mather  and Thurston \cite{Ma71, Th74}, statement c) below. We state some fundamental results.
\begin{itemize}
\vskip .025in
\item [a)] $\Homeo(\S1)_\tau$ is homotopy equivalent to $\S1$. 
\item[b)] The classifying space $B\Homeo(\S1)_\tau$
is a $K(\Z, 2)$, and its  cohomology  is a polynomial algebra over $\mathbb{Z}$ on a generator $\E_\tau$ of degree $2$, called the {\it universal Euler class}.
\item[c)] (Mather--Thurston) The identity $id: \Homeo\S1 \to \Homeo(\S1)_\tau$
induces an algebra isomorphism $id^*:  H^*({B}\Homeo(\S1)_\tau;\Z) \to H^*({B}\Homeo\S1;\Z).$
\item[d)] Any inclusion  $\iota: \mathbb{Z}_m \hookrightarrow \Homeo\S1$ induces an epimorphism of polynomial algebras
$\iota^*:H^*({B}\Homeo\S1; \mathbb{Z}) \to H^*({B}\mathbb{Z}_m; \mathbb{Z}).$ 
 \end{itemize}

A discussion of these results is included in \cite{JJR20} and references therein. See \cite{Th74} for a general version of c) and Section 4.1 of this paper for a specific proof in the case of homeomorphisms of the circle.  In what follows we will not distinguish between the homology of a group and the homology of its classifying space. 

\begin{defn} 
The {\it  universal discrete Euler class $\mathbf{E}$} in $H^2 (\Homeo\S1;\mathbb{Z})\cong\Z$ is the pullback by $id: \Homeo\S1 \to \Homeo(\S1)_\tau$  of the universal Euler class $\E_\tau$.
\end{defn}
Notice that the  $n$-th power of the universal discrete Euler  class $\mathbf{E}^n$  is a generator of   $H^{2n} (\Homeo\S1;\mathbb{Z})\cong\Z$ for $n\geq 1$, hence a  non-trivial torsion-free cohomology class.

\subsection{Homology of a group derived from an action}\label{GROUPCOHO}

Suppose a group $G$ acts on a set $S.$ Then it acts on $S^{\infty}$ 
which is the infinite simplex on the set.
For each $p \geq 0$ the action on the $p$-simplices $S^{\infty}_p$ gives rise to a groupoid $\Lambda_p{G}$ whose objects in dimension $p$ are 
 the $p$-simplices, whose morphisms 
are $(g,\sigma)\in G \times S^{\infty}_p ,$ and whose source and target maps are 
$s(g,\sigma) =  \sigma$ and 
$t(g,\sigma) = g(\sigma).$ The composition of morphisms $(g, \sigma)$ and $(f,g(\sigma))$ 
is $(fg, \sigma)$.  
We will denote a morphism $f$ with source 
$\sigma$ by $\sigma \cdot f,$ a composable pair by $\sigma \cdot( f, g)$ where $g$ acts on $\sigma$ and $f$ on $g(\sigma)$, and so on.

The action of $G$ gives $\Lambda_pG$ the structure of a simplicial groupoid and extending by nerves in the $q$ direction produces a bi-simplicial set $\Lambda G = \Lambda_{p,q} G.$  The horizontal simplicial set for fixed $q$ is the simplical set associated to the simplicial complex $G^q \times S^{\infty}.$ The simplicial complex $S^{\infty}$ is contractible which implies that the realization of $\Lambda G$ is homotopically equivalent to $BG.$

The bi-simplical set $\Lambda G = \Lambda_{p,q} G$ gives rise to a double chain complex ${\mathcal{C}}={\mathcal{C}}_{p,q}$ which computes the homology of $G.$  Each horizontal simplicial set of $\Lambda G$ is associated to a simplicial complex so oriented chains can be used to form a double chain group.  For each fixed $q$ let the chain complex ${\mathcal{C}}_{*,q}$ be the classical oriented chain complex of the simplicial complex $G^q \times S^{\infty}.$  Then the free abelian group $\mathcal{C}_{p,q}$ is generated by chains of the form 
$[v_0,\ldots,v_p]\cdot(f_1,\ldots,f_q)$  where $[v_0,\ldots,v_p]$ denotes an oriented $p$-simplex of $S^\infty$ and $f_1,\ldots,f_q\in G$.

This double chain complex determines a total complex $T{\mathcal{C}}$ given by $T{\mathcal{C}}_n = \bigoplus_{p+q=n}{\mathcal{C}}_{p,q}$ 
with differential $\partial = \partial^h +\partial^v$,  where the horizontal  and vertical boundary homomorphisms are given on summands 
by $\partial^h: {\mathcal{C}}_{p, q} \to {\mathcal{C}}_{p-1, q},$ $\partial^v: {\mathcal{C}}_{p, q} \to {\mathcal{C}}_{p, q-1}$  and satisfy
$\partial^h \partial^v + \partial^v \partial^h =0.$ Then the homology of the total complex $H_{*}(T{\mathcal{C}})$ is $H_{*}(G).$

We define the {\it orbit chain complex  $H_{0}^v{\mathcal{C}}$ of the action of $G$}   by taking  $p \mapsto H_{0}^v({\mathcal{C}}_{p,*}).$ 
It is the chain complex at $q=0$ in the ${\mathcal{E}}^1$-term of the spectral sequence obtained by computing the vertical homology of 
${\mathcal{C}}_{p,q}.$ Explicitly $H_{0}^v({\mathcal{C}}_{p,*})$ is the 
free abelian group on the orbits of oriented $p$-simplices under the action of $G.$ 

There is a chain map ${\rm q}:T{\mathcal{C}} \to H_{0}^v{\mathcal{C}}$ defined as follows.  
Given  $c = c_0 + c_1 + \cdots +c_p \in T{\mathcal{C}}_p,$ with $c_i \in {\mathcal{C}}_{p-i,i},$
let $${\rm q}(c) = [c_0]^v\in H_{0}^v({\mathcal{C}}_{p,*}),$$ where $[c_0]^v$ denotes the vertical homology class of the chain $c_0 \in {\mathcal{C}}_{p,0}.$ We  we call ${\rm q}$ {\it the orbit chain map of the action of $G$.}

The homology of a group $G$  computed using the total complex $T{\mathcal{C}}$ is related to its cohomology by
the Universal Coefficient Theorem:
 \begin{equation}\label{UCT} 0 \to Ext(H_{p-1}(T{\mathcal{C}}), \mathbb{Z}) {\buildrel {\alpha} \over {\longrightarrow}} H^{p}(T{\mathcal{C}}; \mathbb{Z}) {\buildrel {\beta} \over {\longrightarrow}} Hom(H_{p}(T{\mathcal{C}}), \Bbb Z) \to 0.  \end{equation} 
A non-trivial class in $H^{p}(T{\mathcal{C}}; \Bbb Z)$ lies either in $Ext(H_{p-1}(T{\mathcal{C}}), \Bbb Z)$ or  in $Hom(H_{p}(T{\mathcal{C}}), \Bbb Z).$ If the class is infinite cyclic it can be found in the $Hom$-term; if it is finite cyclic it can be found in $Ext$-term.

\section{The mapping class group and its actions}\label{MCG}

 Consider a closed oriented surface $\Sigma_g$ of genus $g \geq 1,$ and let $z \in \Sigma_g.$ The {\it  mapping class group 
$\Gamma_g^1$} is the group of orientation preserving homeomorphisms of $\Sigma_g$ which fix $z,$ modulo isotopies which fix $z.$
For an orientation preserving homeomorphism $f$ of $\Sigma_g$ such that $f(z)=z$,  
let $ f_*$ denote the induced automorphism of $\pi_1(\Sigma_g, z).$  The assignment $[f] \mapsto f_*$ gives a well-defined monomorphism from $\Gamma_g^1$ to  the automorphism group $Aut(\pi_1(\Sigma_g, z))$, by the Dehn–Nielsen–Baer theorem.

\begin{subsection}{The Nielsen action of $\Gamma_g^1$}\label{Nielsen}

Consider $\pi_1(\Sigma_g, z)$ presented as the free group on $2g$ elements ${{\rm a}}_0, {{\rm a}}_1, {{\rm a}}_3 ,{{\rm a}}_4,\ldots,{{\rm a}}_{2g-1}$
with the relation 
\begin{equation}\label{fundamentalRelation}
 {{{\rm a}}}_0 \cdots {{{\rm a}}}_{2g-1} = {{{\rm a}}}_{2g-1} \cdots {{{\rm a}}}_{0}.
 \end{equation}

Let us place a $4g$-gon in the plane, with directed edges labeled counterclockwise by 
${{\rm a}}_0, {{\rm a}}_1,...,{{\rm a}}_{2g-1}, {{\rm a}}_0^{-1}, {{\rm a}}_2^{-1},...,{{\rm a}}_{2g-1}^{-1}$
in such a way that the point at the beginning of the edge ${{\rm a}}_i$ is at the origin and the polygon
generates a tiling of the hyperbolic disk $\mathbb{D}$ for $g>1,$ and of the Euclidean plane when  $g=1$.
The resulting tiling will have $4g$ polygons forming a wreath at the origin. There will be $4g$ rays 
emanating from the origin, $2g$ labeled by the ${{\rm a}}_i$'s
and $2g$ labeled by their inverses which are antipodal to the ${{\rm a}}_i$'s. Figure \ref{fig:genus2} illustrates the case $g=2.$

The {\it Nielsen action of $\Gamma_g^1$ on $\mathbb{S}^1$} is given by a monomorphism $\rho: \Gamma_g^1\hookrightarrow \Homeo\S1$ that we recall next, (see for example \cite{FM12} for more details). For $g\geq 2$, construct the universal cover $\D \to \Sigma_g$ with $\D$ tiled as described above, and so that the origin in $\D$
is mapped to the marked point  $z\in\Sigma_g$. Each $\gamma\in\pi_1(\Sigma_g,z)$ acts on $\D$ as a hyperbolic isometry with a unique translation axis with forward endpoint $\gamma_{\infty}\in\partial \D \approx \S1$. Since the action of $\pi_1(\Sigma_g,z)$ on $\D$ is cocompact, the set $\Gamma_{\infty}=\{\gamma_{\infty}:\gamma\in\pi_1(\Sigma_g,z)\}$ is dense in $\partial \D$. Hence,  any automorphism $\phi$ of $\pi_1(\Sigma_g, z)$ induces a homeomorphism $\partial \phi$ of $\partial \D\approx \S1$ that takes $\gamma_{\infty}\in\Gamma_{\infty}$ to $\phi(\gamma)_{\infty}$. Pre-composing with the action of ${\Gamma}_g^1$ on $\pi_1(\Sigma_g, z)$ gives the homomorphism $\rho: \Gamma_g^1\rightarrow \Homeo\S1$ which turns out to be injective. 

For $g=1$, the group $\Gamma_1^1\cong \SL(2,\mathbb{Z})$ acts faithfully on rays starting at the origin in the Euclidean plane and $\rho$ is the corresponding monomorphism.

Notice that for each generator $\rm a_i$ of $\pi_1(\Sigma_g,z)$, the corresponding translation axis is a geodesic line through the origin of $\D$ that connects two antipodal points of $\partial \D\approx \mathbb{S}^1$. We label the forward endpoint of the traslation axis by $\rm a_i$  and the backward endpoint  by $\rm a_i^{-1}$.  These points are ordered clockwise in $\Bbb S^1$ by
\begin{equation}\label{order1}
{{\rm a}}_0 < {{\rm a}}_{1}^{-1} < {{\rm a}}_{2}  < \cdots < {{\rm a}}_{2g-1}^{-1} < {{\rm a}}_{0}^{-1} < {{\rm a}}_1 <  {{\rm a}}_{2}^{-1} <\cdots < {{\rm a}}_{2g-2}^{-1} < {{\rm a}}_{2g-1}
\end{equation}
Figure \ref{fig:genus2} illustrates this ordering for the case $g=2$.

\begin{figure}[htbp]
 \begin{centering}     \hskip .5in
 \includegraphics[width= 3.5in]{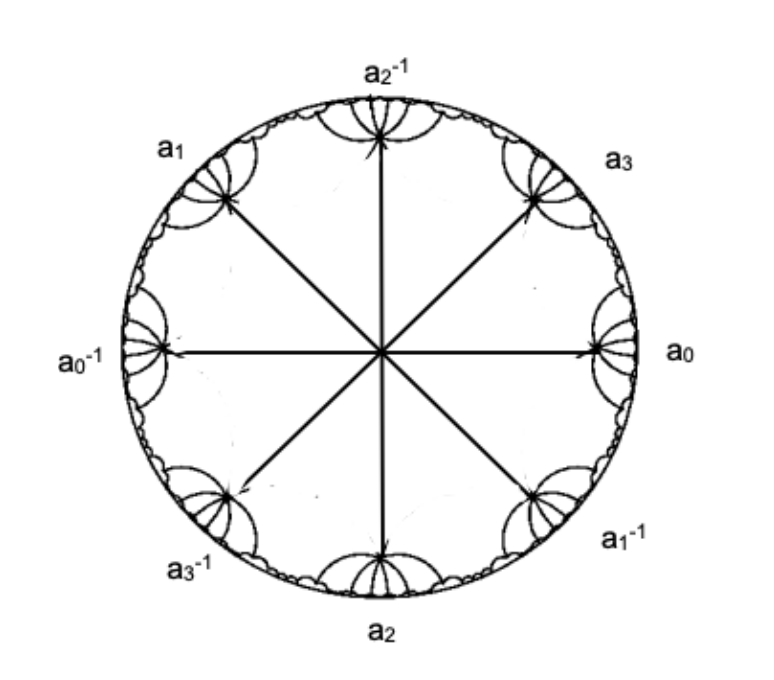} 
 \caption{\small Tiling of the hyperbolic disk $\mathbb{D}$ for genus $g=2$.} \label{fig:genus2}
\end{centering}
\end{figure}

\end{subsection}

\begin{subsection}{ The inversive action of $\Gamma_g^1$}
Let $w$ be an element of $\pi_1(\Sigma_g, z).$ The equivalence relation determined by the relation $w \sim w^{\pm 1}$
divides $\pi_1(\Sigma_g, z)$ into equivalence classes which we refer to as {\it inversives}. The 
action of $\Gamma_g^1$ on $\pi_1(\Sigma_g, z)$ preserves these equivalence classes and induces an action  on the set $\Bbb I$ of inversives which we call the {\it inversive action of $\Gamma_g^1$}.  We denote the inversive of ${\rm a}_i$ by $A_i.$ 

The Nielsen and  inversive actions determine bi-simplicial sets, double chain complexes, and a total complex.
We assume in the remainder of the paper that $\Gamma_g^1$ acts on the orbit of a point.  We choose the point to be ${\rm a}_0$ for the Nielsen action and $A_0$ for the inversive action.  We will, when appropriate, denote constructions for the actions by subscripts $\Bbb N$ and $\Bbb I.$  There is a bi-simplicial map of bi-simplicial sets $\nu: \Lambda_{\Bbb N} \Gamma^g_1 \to   \Lambda_{\Bbb I} \Gamma^g_1,$
induced by the quotient map $\pi_1(\Sigma_g,z)\twoheadrightarrow \mathbb{I}$. 
\begin{prop}\label{EQUIV} The map $\nu$ induces a chain map homology isomorphism
from $T{\mathcal{C}}_{\Bbb N}$ to $T{\mathcal{C}}_{\Bbb I}.$ 
\end{prop}
\begin{proof} 
Each of the bi-simplicial sets realizes as the product of a $K(\Gamma_g^1,1)$ and an infinite simplex. On the realizations the map induced by $\nu$ is the product of the identity and a homotopy equivalence, and therefore a homology equivalence on their associated total complexes.
\end{proof}

\begin{prop}\label{INVERSE} The chain map $ \bar \nu_*:H_0^v{\mathcal{C}}_\mathbb{N}\rightarrow H_0^v{\mathcal{C}}_\mathbb{I}$ induced by $\nu$ is an isomorphism in homology
and cohomology. 
\end{prop}
\begin{proof} We find an inverse chain map $j_*:T{\mathcal{C}}_\mathbb{I} \rightarrow T{\mathcal{C}}_\mathbb{N}$ to  $\nu_*,$
in homology. By functoriality it then induces the inverse $\bar j_*$ to $H_0^v{\mathcal{C}}_\mathbb{N}\rightarrow H_0^v{\mathcal{C}}_\mathbb{I}$ in homology.
For each inversive $X$ choose a  representative point $x$ on the circle, and extend the choice to $n$-tuples and thereby identify the infinite simplex on inversives as a sub-complex of the infinite simplex on the points of ${\mathbb S}^1.$ Since an element of the mapping class group acts component-wise this assignment commutes with the action of $\Gamma_g^1$ for it take the orbits of inversives $(X_1,...,X_n)$ to the orbits of representatives $(x_1,...,x_n).$ 
So $j_*$ preserves the  bi-simplicial structure and induces a chain map $j_*:T{\mathcal{C}}_\mathbb{I} \rightarrow T{\mathcal{C}}_\mathbb{N},$  By construction $j_*$ is a right inverse to $\nu_*,$ but it is then a two-sided inverse in homology since $\nu_*$ is an isomorphism in homology. The associated cochain complexes are also equivalent so we have an isomorphism in cohomology as well.
\end{proof}

\begin{rem}\label{chooseInv} Given the inversives $A_0, A_1,...,A_{2g}$ we choose $j_*:T{\mathcal{C}}_\mathbb{I} \rightarrow T{\mathcal{C}}_\mathbb{N}$ so that $A_i  \mapsto a_i$ for $i$ even, and $A_i  \mapsto a^{-1}_i$ for $i$ odd. 
\end{rem}

\end{subsection}

\begin{subsection}{Characteristic morphisms}\label{CHAR}
Consider the presentation of $\pi_1(\Sigma_g, z)$ from Section \ref{Nielsen} and let ${{\rm a}}^{-1}_{2g} = {{\rm a}}_0 {{\rm a}}_1\cdots {{\rm a}}_{2g-1}$.   As before, we label the forward endpoint by $\rm a_{2g}$  and the backward endpoint by $\rm a_{2g}^{-1}$ in $\partial \D$ of the  translation axis of ${\rm a_{2g}}$.  Whether considering the ${{\rm a}}_i^{\pm}$ as elements of  $\pi_1(\Sigma_g, z)$ or points in $\S1$, we see that there exist mapping classes, which are homeomorphisms of the circle, $S$ of order $4g$ and $T$ of order $2g+1$ given by
$$ S: {{\rm a}}_0 \to {{\rm a}}_{1 } \to \cdots \to {{\rm a}}_{2g-1} \to {{\rm a}}_0^{-1} \to \cdots \to {{\rm a}}_{2g-1}^{-1} \to {\rm a}_0
\hskip .5in T: {{\rm a}}_0 \to {{\rm a}}_1  \cdots \to {{\rm a}}_{2g-1} \to {{\rm a}}_{2g} \to \rm a_0 .$$
 The mapping class $d_{2g}={S}^{-1}{T}$ acts on $\pi_1(\Sigma_g,z)$ by taking 
$ {{\rm a}}_{2g}$ to ${{\rm a}}^{-1}_{2g-1}$  while fixing all ${{\rm a}}_i$ and their inverses for $0 \leq i \leq 2g-2.$ More generally, let $d_{2g-k}:={T}^{-k}({S}^{-1}T){T^k}\text{\  for \ } 0 \leq k \leq 2g.$ It takes ${{\rm a}}_{2g-k}$ to ${{\rm a}}^{-1}_{2g-k-1}$ mod$(2g+1)$, and keeps
all ${{\rm a}}_i$ and their inverses, with indices other than $2g-k$ and $2g-k-1,$ fixed. 

\begin{rem}
Notice $d_0$ maps ${\rm a}_0$ to  ${\rm a}^{-1}_{2g}$ and keeps all other ${\rm a}_i$ fixed. The ordering (3) must be preserved so, since $a_0$ is between
$a_{2g-1}$ and ${\rm a}^{-1}_1$  the image of  ${\rm a}_{0}$ is either between ${\rm a}_0$ and ${\rm a}_{2g-1}$ or between ${\rm a}_0$ and ${\rm a}^{-1}_1.$ Moreover $d_{2g}$ maps ${\rm a}^{-1}_{2g}$ to ${\rm a}_{2g-1}$ keeping all remaining points fixed, so only the first of the two possibilities can hold.
From the ordering (\ref{order1})  and this observation we have that the points, including $a_{2g},$ are ordered clockwise in $\Bbb S^1$ as follows.
\begin{equation}\label{order}
{{\rm a}}_0 < {{\rm a}}_{1}^{-1} < {{\rm a}}_{2}  < \cdots < {{\rm a}}_{2g-1}^{-1} < {{\rm a}}_{2g} < {{\rm a}}_{0}^{-1} < {{\rm a}}_1 <  {{\rm a}}_{2}^{-1} <\cdots < {{\rm a}}_{2g-2}^{-1} < {{\rm a}}_{2g-1}< {{\rm a}}_{2g}^{-1}
\end{equation}
\end{rem}

\begin{defn}  
For $ 0 \leq j \leq 2g$, we refer to $d_j$ as a {\it characteristic mapping class}. It determines a {\it characteristic morphism} in the double chain complex $\mathcal{C}_{\Bbb I}$ of the inversive action $$ [{ A}_0, { A}_1,...,\hat {{A}}_{j-1}, {A}_{j},...,{A}_{2g}] \cdot d_{j}\in(\mathcal{C_{\Bbb I}})_{2g-1,1}.$$ 
\end{defn}
The mapping classes $S$ and $T$ act on the set $\{A_0,...,A_{2g}\}$  as follows
$$S: A_0 \to A_1 \to \cdots \to A_{2g-1} \to  A_{0}, \hskip .5in   T: A_0 \to A_1 \to A_2 \cdots \to A_{2g-1} \to A_{2g} \to A_{0}.
$$
Furthermore, the characteristic mapping class $d_{2g-k}$ takes $A_{2g-k}$ to $A_{2g-k-1}$ mod$(2g+1)$, and  fixes all the inversives ${A_i}$, with indices other than $2g-k$ and $2g-k-1$.

These mapping classes satisfy the following identities in $\Gamma_g^1$:
\begin{equation}\label{identities}
\begin{array}{c}
d_{2g}d_{2g-1} \cdots d_{0}=(S^{-1}T)(T^{-1}S^{-1}T^2) \cdots (T^{-2g}S^{-1}T^{2g+1})=
(S^{-1})^{2g+1}T^{2g+1} = S^{-1} \\
(d_{2g}d_{2g-1} \cdots d_{0})^{2g} =  S^{2g} \hskip .25in   (d_{2g}d_{2g-1} \cdots d_{0})^{4g} = 1_{\Gamma_g^1}\end{array}
 \end{equation}
Note, the ``order of the action" of $S$ on the $A_i$ is now $2g,$ however that does not affect the order of $S$ in $\Gamma_g^1$
which remains $4g.$

\end{subsection}

\section{The Euler class and its orbit cocycles}\label{Euler}

\begin{subsection}{The universal Euler class and its orbit cocycle}

 The group $G= { Homeo^+\hskip .02in {\S1}}$  acts on the contractible simplicial complex which is the infinite simplex on $\S1.$ In this section we let  $\mathcal{C}={\mathcal{C}}_{p,q}$ be the double chain complex of oriented chains associated to the action, with  $T\mathcal{C}$ the corresponding total complex and $H_0^v{\mathcal{C}}$ the orbit chain complex.  The following homology computation appeared in \cite[Section 5]{Je12} and uses the fact that $Homeo^{+}{\Bbb R}$ is acyclic, \cite{Se78, Mc80}. See also \cite{FMN24} for a new proof.

\begin{thm}\label{HOMOUNiVERSAL} 
Consider $G= { Homeo^+\hskip .02in {\S1}}$  acting on the infinite simplex on $\S1.$ 
\begin{itemize}
\vskip .025in

\item[a)] The homology and cohomology of $H_0^v {\mathcal{C}}$ is 
${\Bbb Z}$ in even dimensions and $0$ in odd.  
A generator 
of $H_{2p}(H_0^v {\mathcal{C}})$ is represented by the cycle
$\mathbf{o}_p:= [{0, {2p}, 1,...,{2p-1}}] - [{ 0, 1,2,...,{2p}}]$ in $H_0^v {\mathcal{C}}$, 
where $ 0 < 1<2< \cdots $ is a countable set of clockwise oriented points of the circle. 
\item [b)] The orbit chain map ${\rm q}:T{\mathcal{C}}  \to H_0^v{\mathcal{C}}$
 induces a homology equivalence.

\end{itemize}
 \end{thm}

\begin{proof}
To prove a) we show that
$H_0^v{\mathcal{C}}$ is given by
$ {\mathbb{Z}} \leftarrow {\Bbb Z}_2 \leftarrow {\Bbb Z} \leftarrow {\Bbb Z}_2 \leftarrow
\cdots.$
Any clockwise ordered $(p+1)$-tuple,
$(x_0,...,x_p),$ of points of the circle
can be mapped to $\{{ 0 , 1, 2,...,p }\}$ in some order. Hence there are at most two
orientation classes of $p$-simplices, one represented by $[{ 0 , 1, 2,..., p }]$ and one
represented by any odd permutation of the vertices.
When $p$ is even, a cyclic permutation of the coordinates of a $p$-simplex determines
a new simplex in the same orientation class,
so that there are two orientation classes and the group of oriented $p$-chains is
isomorphic to $\mathbb{Z},$ with
$+1$ and $-1$ representing the classes $[{0, {2p}, 1,...,{2p-1}}]$ and $[{ 0,
1,2,...,{2p}}]$, respectively.
Note that $[{ 0, 1,2,...,{2p}}] = -[{0, {2p}, 1,...,{2p-1}}].$
When $p$ is odd a cyclic permutation of the coordinates of a $p$-simplex determines a
simplex in the opposite orientation class, so there is one class of order 2, and the $p$-th
chain group is isomorphic to $\mathbb{Z}_2.$ The homology of $H_0^v{\mathcal{C}}$ is ${
\Bbb Z}$ in even dimensions and $0$ otherwise, and is therefore generated in
dimension $2p$ by the class of the cycle ${ \mathbf{o}_{p} }=
2[{0, {2p}, 1,...,{2p-1}}]$ or equivalently $[{0, {2p}, 1,...,{2p-1}}] - [{ 0, 1,2,...,{2p}}].$ The
cohomology of $H_0^v{\mathcal{C}}$ is isomorphic to $\Bbb Z$
in even dimensions and $0$ in odd. This completes the proof of a).

For b) we will use the action of $G$ on ordered simplices as well as on oriented
simplices. The former is ``right for isotropy" and the latter is ``right for orbits".
The homomorphism from ordered chains on a simplicial complex to oriented chains is
defined by associating to an ordered simplex its orientation class. It induces a map from
the double chain complex $\Delta$ of the action of $G$ on ordered simplices to the
double chain complex ${\mathcal{C}}$ of the action of $G$ on oriented simplices which is a
homology equivalence on each horizontal complex. Therefore there is
chain homomorphism
$T{\Delta} \longrightarrow T{\mathcal{C}}$ which is well defined on orbits and functorially
determines a chain homomorphism
$H_0^v\Delta \longrightarrow H_0^v\mathcal{C}.$ A chain inverse $T{\mathcal{C}} \longrightarrow
T{\Delta}$ is defined by totally ordering the vertices and choosing for each oriented
simplex the ordered simplex given by the total ordering. It too is well defined on orbit
complexes so induces a chain inverse $H_0^v{\mathcal{C}} \longrightarrow H_0^v\Delta.$
Consider the commutative diagram of chain complexes
$$T{\Delta} \longrightarrow T{\mathcal{C}}$$
\vskip -.25in
$$ \hskip -.05in \rm q_{\Delta} \downarrow \hskip .5in \downarrow \rm q$$
$$H_0^v\Delta \longrightarrow H_0^v\mathcal{C}.$$
The chain homomorphisms $T{\Delta} \longrightarrow  T{\mathcal{C}}$  and $H_0^v\Delta \longrightarrow H_0^v\mathcal{C}$ induce  homology equivalences. We show that
$\rm q_{\Delta}$ is a homology equivalence, which implies that $\rm q$ is also, and proves b).

The double complex ${ \Delta}_{p,q}$ is the ${\mathcal{E}}^0$-term of the spectral sequence
$${\mathcal{E}}^2_{p,q}(T{\Delta}) = H_{p}^h H_{q}^v({\Delta})
\Rightarrow H_{p+q}({T{\Delta}}).$$
We observe that computing homology vertically then horizontally in the the spectral
sequence gives
\begin{equation*} {\mathcal{E}}_{p,q}^{2} =
\begin{cases} H_{p}(H_0^v{\Delta}) &\text{ if }q=0 \\
0 &\text{ if } q>0. \end{cases}
\end{equation*}
The homology of each vertical groupoid is isomorphic to the direct sum of the homology
of its isotropy subgroups. The isotropy
group of each $p$-simplex is acyclic, for the isotropy group of any point $b$ is
isomorphic to the group of orientation preserving homeomorphisms of $\mathbb S^1 - \lbrace b \rbrace,$ which in turn is isomorphic to $Homeo^{+}{\Bbb R}.$
The isotropy group of a $(p+1)$-element subset of $\mathbb{S}^1$ is isomorphic to a $(p+1)$-fold cartesian product of $Homeo^{+}{\Bbb R},$ so it too is acyclic. This implies that all the homology groups  in ${\mathcal{E}}_{p,q}^{1}$ for $q>0$ are $0.$

It remains to show that the isomorphism between $H_p(T\Delta)$ and
$H_{p}(H_0^v{\Delta})$ is induced by ${\rm q}_\Delta.$ Let  $c_0 \in {\Delta}_{p,0}$
represent a cycle in $H_{p}(H_0^v{\Delta}).$ Because all the vertical groupoids are
acyclic it follows that the successive differentials of the spectral sequence are trivial,
and $c_0$ lifts to a $p$-cycle $c$ in  $T{\Delta}$ (see the introduction to Section \ref{SecTranCycle} for more on the lifting process). Then  ${\rm q}_{\Delta}(c) = [c_0]^v$ as required.  
\end{proof}

\begin{rem} The Universal Coefficient Theorem fails for the chain complex $H_0^v{\mathcal{C}}$ as it is not free. The group $Hom (H_{2p}(H_0^v{\mathcal{C}}),\mathbb{Z})$ is isomorphic to $\Bbb Z$ and generated by the homomorphism $\mathbf{o}^*_{p}$  dual to the {\it universal orbit cycle} $$\mathbf{o}_{p}=[{0, {2p}, 1,...,{2p-1}}] - [{ 0, 1,2,...,{2p}}]=2[{0, {2p}, 1,...,{2p-1}}].$$ On the other hand, the cohomology group $H^{2p}(H_0^v{\mathcal{C}}; \mathbb{Z})$ is isomorphic to $\Bbb Z$ and it is generated by the cocycle $(\mathbf{o}_{p}/2)^*\in Hom(H_0^v({\mathcal{C}}_{2p,*}), \mathbb{Z})$ dual to the chain $$(\mathbf{o}_{p}/2):= [{0, {2p}, 1,...,{2p-1}}].$$ 
\end{rem}

For $p\geq 1$, let us consider the chains  $$(\mathbf{e}_p/2):=[{0, {2p}, 1,...,{2p-1}}]\in {\mathcal{C}}_{2p,0}\text{\ \ and \ \ }\mathbf{e}_p:=[{0, {2p}, 1,...,{2p-1}}] - [{ 0, 1,2,...,{2p}}]=2(\mathbf{e}_p/2).$$ These are a chains in the total complex $T {\mathcal{C}}$ that ``lift"  $(\mathbf{o}_{p}/2)$ and the orbit cycle $\mathbf{o}_p$, respectively, in the sense that  ${\rm q}(\mathbf{e}_p/2)=[\mathbf{e}_p/2]^v=(\mathbf{o}_p/2)$ and ${\rm q}(\mathbf{e}_p)=[\mathbf{e}_p]^v=\mathbf{o}_p$.

\begin{defn} We refer to  $(\mathbf{o}_{p}/2)^*$ as the {\it $p$-th universal orbit cocycle} and to the pull-back cocyle $(\mathbf{e}_p/2)^*:={\rm q}^*(\mathbf{o}_{p}/2)^*$ as the  {\it $p$-th universal Euler cocycle}.
\end{defn}

\begin{prop} The $p$-th universal Euler cocycle $(\mathbf{e}_p/2)^*$ satisfies
$\langle(\mathbf{e}_p/2)^*, \mathbf{e}_p\rangle =2$ and it represents $2{\mathbf E}^p$ in the total complex $T{\mathcal{C}}.$
\end{prop}
\begin{proof}

Since the class ${\mathbf E}^p$ generates the cohomology group $H^{2p}(G;\mathbb{Z})\cong\mathbb{Z}$, the cocycle $(\mathbf{e}_p/2)^*$ represents an integral multiple of ${\mathbf E}^p$. The total complex $T{\mathcal{C}}$ satisfies the Universal Coefficient Theorem (\ref{UCT}) and we know from Theorem \ref{HOMOUNiVERSAL} that ${\rm q}$ induces an isomorphism in homology. Under the composition of isomorphisms 
$$H^{2p}({T\mathcal{C}};\Bbb Z)\xrightarrow{\beta} Hom(H_{2p}(T{\mathcal{C}}), \Bbb Z)\xrightarrow{ ({\rm q}^*)^{-1}} Hom(H_{2p}(H_0^v{\mathcal{C}}), \Bbb Z)$$
the cohomology class ${\mathbf E}^p$  corresponds to $\mathbf{o}_p^*$ the homomorphism dual to the orbit cycle $\mathbf{o}_p$, which generates $Hom(H_{2p}(H_0^v{\mathcal{C}}), \Bbb Z)\cong \Bbb Z$. On the other hand, 
$$\langle(\mathbf{e}_p/2)^*, \mathbf{e}_p\rangle =\langle{q}^*(\mathbf{o}_{p}/2)^*, \mathbf{e}_p\rangle=\langle(\mathbf{o}_{p}/2)^*, { q}(\mathbf{e}_p)\rangle=\langle(\mathbf{o}_{p}/2)^*, \mathbf{o}_{p}\rangle=2.$$
Therefore the composition $({\rm q}^*)^{-1}\circ\beta$  takes the cohomology class represented by the cocycle $(\mathbf{e}_p/2)^*$ to $2\mathbf{o}_p^*$. It follows that the cocycle $(\mathbf{e}_p/2)^*$ corresponds to $2{\mathbf E}^p$.
\end{proof}

\end{subsection}

\begin{subsection}{The Euler class of $\Gamma_g^1$ and its orbit cocycles }\label{OrbitCocycles}

 We now consider the Nielsen and inversive actions of $\Gamma_g^1$ to study the behavior of the $g$-th power of the Euler class of $\Gamma_g^1$.
\begin{defn}\label{EulerMCG}
    The Nielsen action $\rho:\Gamma_g^1\hookrightarrow \Homeo\S1$ pulls-back the powers  $\mathbf{E}^p$
producing classes $\E^p \in H^{2p}({\Gamma}_g^ 1;\Bbb Z)$ for all  $p\geq 1$. We refer to $\E$ as the {\it  Euler class of $\Gamma_g^1$}.
\end{defn}

Consider the following diagram of chain complexes:
\begin{equation}\label{diagram}
\begin{array}{ccccc}
T\mathcal{C}_\mathbb{I} & \rightarrow & T\mathcal{C}_\mathbb{N}  & \rightarrow & T\mathcal{C} \\
\downarrow & & \downarrow & &\downarrow \\
 H_0^v{\mathcal{C}}_\mathbb{I}  &\rightarrow & H_0^v{\mathcal{C}}_\mathbb{N} & \rightarrow  & H_0^v{\mathcal{C}}
\end{array}
\end{equation}
The vertical arrows are the corresponding orbit chain maps ${\rm q}_\mathbb{I}$, ${\rm q}_\mathbb{N}$ and ${\rm q}$; the horizontal arrows on the right are induced by the Nielsen action $\rho$; the horizontal arrows on the left are induced by the map $j_*$ defined in the proof of Proposition \ref{INVERSE}. We use this diagram and the  universal cocycles to define cocycles in the total complexes associated to the Nielsen and inversive actions.

\begin{defn}We refer to  $({o}_{g}/2)^*_{\Bbb N}:=\rho^*((\mathbf{o}_g/2)^*)$ as the {\it $g$-th orbit cocycle for the Nielsen action} and to ${({e}_g/2)}^*_{\Bbb N}:=\rho^*((\mathbf{e}_g/2)^*)$ as the {\it $g$-th Euler cocycle for the Nielsen action}.
\end{defn} 

From the definition of the Euler class, we have that the cocycle $(e_g/2)^*_{\Bbb N}$  represents the cohomology class $2\E^g\in H^{2g}(\Gamma_g^1;\mathbb{Z})$ in the total complex of the Nielsen action. 

Consider the {\it orbit chains} in the complex $H_0^v{\mathcal{C}}_\mathbb{N}$
$${(o_g/2)}_{\Bbb N} := [{{{\rm a}}^{-1}_0, {{\rm a}}^{-1}_{2g}, {{\rm a}}_1,{{\rm a}}^{-1}_{2},...,{{\rm a}}_{2g-1}}]\text{\ \ \ and \  \  \  }$$
$${(o_g)}_{\Bbb N} := [{{{\rm a}}^{-1}_0, {{\rm a}}^{-1}_{2g}, {{\rm a}}_1,{{\rm a}}^{-1}_{2},...,{{\rm a}}_{2g-1}}] - [{{\rm a}}^{-1}_0, {{\rm a}}_1, {{\rm a}}^{-1}_2,...,{{\rm a}}_{2g-1}, {{\rm a}}_{2g}^{-1}]=2{(o_g/2)}_{\Bbb N},$$ 
and the corresponding ``lifts'' to the total complex $T\mathcal{C}_{\mathbb{N}}$ 
$${(e_g/2)}_{\Bbb N} := [{{{\rm a}}^{-1}_0, {{\rm a}}^{-1}_{2g}, {{\rm a}}_1,{{\rm a}}^{-1}_{2},...,{{\rm a}}_{2g-1}}]\text{\ \ \ and \  \  \  }$$
$${(e_g)}_{\Bbb N} := [{{{\rm a}}^{-1}_0, {{\rm a}}^{-1}_{2g}, {{\rm a}}_1,{{\rm a}}^{-1}_{2},...,{{\rm a}}_{2g-1}}] - [{{\rm a}}^{-1}_0, {{\rm a}}_1, {{\rm a}}^{-1}_2,...,{{\rm a}}_{2g-1}, {{\rm a}}_{2g}^{-1}]=2{(e_g/2)}_{\Bbb N},$$ 
which are chains in  $({\mathcal{C}_{\Bbb N}})_{2g,0}$. Since the elements $\{{\rm a}_i^{\pm 1}\}$ are cyclically ordered as indicated in (\ref{order}), these chains satisfy $\rho_*({(o_g)}_{\Bbb N})=\mathbf{o}_g$ and $\rho_*({(e_g)}_{\Bbb N})=\mathbf{e}_g$. Moreover, evaluation on the cocycles for the Nielsen action gives  $$\langle (e_g/2)_{\Bbb N}^*, {(e_g)}_{\Bbb N} \rangle=\langle (o_g/2)_{\Bbb N}^*, {(o_g)}_{\Bbb N}\rangle=2 \text{ \ and \ } \langle (e_g/2)_{\Bbb N}^*, {(e_g/2)}_{\Bbb N} \rangle=\langle (o_g/2)_{\Bbb N}^*, {(o_g/2)}_{\Bbb N}\rangle=1.$$

\begin{defn}
The {\it $g$-th orbit cocycle for the inversive action of $\Gamma_g^1$} is $({o}_{g}/2)^*_{\Bbb I}: = \bar j^*({o}_{g}/2)^*_{\Bbb N} $ and the {\it $g$-th Euler cocycle for the inversive action} is 
$({e}_{g}/2)_{\Bbb I}^*:={\rm q}_{\mathbb{I}}^*(({o}_{g}/2)^*_{\Bbb I}).$ 
\end{defn}

Consider the {\it orbit chains} in the complex $H_0^v{\mathcal{C}}_\mathbb{I}$
$${(o_g/2)}_{\Bbb I} := [A_0, A_{2g}, A_{1},\ldots,A_{2g-1}]\text{\ \ \ and \  \  \  }$$
$${(o_g)}_{\Bbb I} := [A_0, A_{2g}, A_{1},\ldots,A_{2g-1}] - [A_0, A_1, A_2,\ldots,A_{2g-1}, A_{2g}]=2{(o_g)}_{\Bbb I},$$ 
and the corresponding ``lifts'' to the total complex $T\mathcal{C}_{\mathbb{I}}$ 
$${(e_g/2)}_{\Bbb I} := [A_0, A_{2g}, A_{1},\ldots,A_{2g-1}]\text{\ \ \ and \  \  \  }$$
$${(e_g)}_{\Bbb I} := [A_0, A_{2g}, A_{1},\ldots,A_{2g-1}] - [A_0, A_1, A_2,\ldots,A_{2g-1}, A_{2g}]=2{(e_g)}_{\Bbb I},$$ 
which are chains in  $({\mathcal{C}_{\Bbb I}})_{2g,0}$.  Notice that by the choices in Remark \ref{chooseInv},  these chains satisfy $\bar j_*({(o_g)}_{\Bbb I})={(o_g)}_{\Bbb I}$ and $j_*({(e_g)}_{\Bbb I})={(e_g)}_{\Bbb I}$.   Evaluation on the cocycles for the inversive action gives  $$\langle (e_g/2)_{\Bbb I}^*, {(e_g)}_{\Bbb I} \rangle=\langle (o_g/2)_{\Bbb I}^*, {(o_g)}_{\Bbb I}\rangle=2 \text{ \ and \ } \langle (e_g/2)_{\Bbb I}^*, {(e_g/2)}_{\Bbb I} \rangle=\langle (o_g/2)_{\Bbb I}^*, {(o_g/2)}_{\Bbb I}\rangle=1.$$

Hence, the arrows in diagram (\ref{diagram}) relate the cocycles and chains that we have constructed in the total complexes and the orbit chain complexes as follows.

\begin{multicols}{2}
\begin{equation*}
\begin{array}{ccccc}
(e_g/2)^*_{\Bbb I} & \mapsfrom & (e_g/2)^*_{\Bbb N} & \mapsfrom &(\mathbf{e}_g/2)^*\\
\upmapsto & & \upmapsto & & \upmapsto\\
(o_g/2)^*_{\Bbb I} & \mapsfrom & (o_g/2)^*_{\Bbb N} & \mapsfrom &(\mathbf{o}_g/2)^*
\end{array}
\end{equation*}

\begin{equation*}
\begin{array}{ccccc}
(e_g)_{\Bbb I} & \mapsto & (e_g)_{\Bbb N} & \mapsto &\mathbf{e}_g\\
\downmapsto & & \downmapsto & & \downmapsto\\
(o_g)_{\Bbb I} & \mapsto & (o_g)_{\Bbb N} & \mapsto &\mathbf{o}_g
\end{array}
\end{equation*}
\end{multicols}

\vskip .25in
From the constructions in this subsection, Definition \ref{EulerMCG} and  Proposition \ref{INVERSE}, we conclude the following.

\begin{prop}\label{InvCocycle} The $g$-th  Euler cocycle for the inversive action $({{e}_g/2})^*_{\Bbb I}$ represents, in the total complex $T{\mathcal{C}}_{\mathbb{I}}$, the cohomology class $2\E^g\in H^{2g}(\Gamma_g^1;\mathbb{Z})$.  Moreover, it satisfies $$\langle ({{e}_g/2})^*_{\Bbb I}, ({{e}_g})_{\Bbb I}\rangle=2\text{\ \ and \ \ } \langle (e_g/2)_{\Bbb I}^*, {(e_g/2)}_{\Bbb I} \rangle=1.$$
\end{prop}

\end{subsection}

\section{The transition cycle}\label{SecTranCycle}

Consider the orbit chain map ${\rm q}:T{\mathcal{C}}  \to H_0^v{\mathcal{C}}$ of a given action. 
We attempt to ``lift" a specific chain $\mathbf{o}$ in $H_0^v{\mathcal{C}}$ 
to a cycle $c$ in the total complex  $T{\mathcal{C}}$ in the sense that ${\rm q}(c)=\mathbf{o}$.  

To describe the lifting process in general 
consider an $n$-chain $\mathbf{o}$ in $H_0^v{\mathcal{C}}$ and let $c_0 \in {\mathcal{C}}_{n,0}$ be a chain such that ${\rm q}(c_0)=[c_0]^v=\mathbf{o}$.  The horizontal boundary $b_0\in {\mathcal{C}}_{n-1,0}$  of $c_0$ is a horizontal cycle. Suppose it bounds vertically; let $c_1\in {\mathcal{C}}_{n-1,1}$ be the vertical boundary of $b_0$. Continuing in this manner, construct $b_{k} \in {\mathcal{C}}_{n-k,k-1}$ and, if possible, $c_{k} \in {\mathcal{C}}_{n-k,k}.$  If the lifting process leads to the construction of $c_{n} \in {\mathcal{C}}_{0,n}$, then $c =c_0 + c_1 + \cdots +  c_{n}$  is an $n$-cycle in the total chain complex  $T\mathcal{C}$ and ${\rm q}(c)=[c_0]^v=\mathbf{o}$. If the lifting terminates at an earlier stage $k,$ that is $b_{k}$ fails to lift further,  then the construction produces an $n$-boundary $b_0 +b_1 + \cdots +  b_{k}$. Moreover,  ${\rm q} (c_0+c_1+\cdots + c_k)=[c_0]^v=\mathbf{o}$, but the chain $c_0+c_1+\cdots + c_k\in T\mathcal{C}_n$ is not a cycle. 

In the universal setting the orbit cycle $\mathbf{o}_g$ lifts to a $2g$-cycle $c$ in $T{\mathcal{C}}$, since ${\rm q}$ induces a homology isomorphism. In contrast, the chain $(\mathbf{o}_g/2)=[{0, {2g}, 1,...,{g-1}}]$ lifts to the chain  $(\mathbf{e}_g/2)\in \mathcal{C}_{2g,0}$, but we cannot extend our construction to even $c_1,$ for its chain boundary is a sum of an odd number of distinct faces so orbits cannot cancel. 

 In this section we consider the inversive action of the mapping class group $\Gamma_g^1$ and attempt to lift the $2g$-chain $$o_g:=(o_g)_{\mathbb{I}}=[A_0, A_{2g}, A_{1},\ldots,A_{2g-1}] - [A_0, A_1, A_2,\ldots,A_{2g-1}, A_{2g}]$$ in $H_0^v{\mathcal{C}}_{\mathbb{I}}$ to a $2g$-cycle in  $T\mathcal{C}_{\mathbb{I}}$. 
We will show that the lifting will terminate with an element $t=b_1 \in ({\mathcal{C}}_{\mathbb{I}})_{2g-2,1}$ which we call the {\it transition cycle}. In Section \ref{SecTrans} we construct the cycle $t$ and analyze it combinatorially. This will lead us to a proof in Section \ref{SecHolonomy} that  $t$ is an obstruction to the lifting and then in Section \ref{SecTorsion} to a calculation of the torsion of $\E^g.$

After this point we suppress the subscript $\Bbb I$ when the inversive context is clear.

\begin{subsection}{Construction of the transition cycle}\label{SecTrans}

We apply the lifting process to the orbit cycle $o_g=(o_g)_{\mathbb{I}}$ of the inversive action. We start by taking $c_0= e_g=[A_0, A_{2g}, A_{1},\ldots,A_{2g-1}] - [A_0, A_1, A_2,\ldots,A_{2g-1}, A_{2g}]$.

\begin{prop}\label{CYCLE} There is a chain $c_1 \in {\mathcal{C}}_{2g-1,1}$ so that $\partial^v(c_1) = \partial^h(e_g).$
\end{prop}

\begin{proof}
We group the horizontal face maps of $e_g$ into pairs that lie in the same orbit. but have opposite signs.
Let $Q=  [ {A_0, A_{2g}, A_1,...,A_{2g-1}}]$ and $R = [ {A_0, A_1, A_2,...,A_{2g}}],$ then, by definition, $e_g=Q-R$. Notice that 
$\partial_0 Q - \partial_1 Q =   [{\hat{A}_0, A_{2g}, A_1,...,A_{2g-1}}] - [{A_0, {\hat {A}_{2g}},...,A_{2g-1}}].$

The two simplices on the right hand side are in the same orbit since 
the characteristic mapping class $d_0$ maps $[{A_0, A_1,...,A_{2g-1}}]$ to
 $[{A_{2g}, A_1,...,A_{2g-1}}],$ 
then $\partial_0 Q - \partial_1 Q= 0$ in the homology of the orbit complex.
For $3 \leq k \leq 2g-1,$ and $k$ odd consider
$$\partial_{k-1} Q - \partial_{k} Q=[{A_0, A_{2g}, A_1,...,\hat {A}_{k-1},...,A_{2g-1}}] - [{A_0, A_{2g},A_1, A_2,...,\hat {A}_{k-2},...,A_{2g-1}}].$$
The two simplices are  in the same orbit of $d_k$ and $\partial_{k-1} Q - \partial_{k}Q= 0$.
There is one face of $Q$ remaining which is $\partial_{2g}Q = [ {A_0, A_{2g}, 1,...,\hat {A}_{2g-1}}].$
Now consider $R$, and notice that the difference  
$\partial_{k-1}  R - \partial_{k} R=[{A_0, A_1,...,\hat {A}_{k},...,A_{2g-1},A_{2g}}] - [{A_0, A_1, A_2,...,\hat {A}_{k-1},...,A_{2g}}]$
consists  of simplices  in the same orbit of $d_{k-1}$ for $1 \leq k \leq 2g-1,$ and $k$ odd. Then $\partial_{k-1}  R - \partial_{k} R=0$. The remaining face is $\partial_{2g}R = [ {A_0, A_1,...,A_{2g-1}, \hat {A}_{2g}}].$
The two remaining faces of ${{e}_g},$   
$[ {A_0, A_{2g}, A_1,...,\hat {A}_{2g-1}}] - [ A_0, A_1,...,A_{2g-1}, \hat {A}_{2g}],$
are in the orbit of $d_{2g},$ which can be seen by cycling $A_{2g}$ to the last slot.
So $\partial^h(e_g)$ bounds vertically. \end{proof}

Explicitly:  $c_1=  \sum_{ i=0}^{2g}(-1)^{i}[ {A_0,...,\hat {A}_{i}, A_{i+1},...,A_{2g}}]d_{i+1,}\text{\ \ \ \ where \ \ }i \equiv n {\rm \hskip .025in mod}(2g+1). $

\end{subsection}

\begin{defn}
The {\it transition cycle} is the chain $t = \partial^h(c_1) \in {\mathcal{C}}_{2g-2, 1}.$    
It is  a cycle, vertically and horizontally, and a boundary horizontally. It breaks into two parts as follows:
\begin{equation}\label{transitionCycle}
t=\sum_{i >j }(-1)^{i+j}[{{A_0,...,\hat {A}_ j,...,\hat {A}_i,...,A_{2g}}}]d_{i+1} +
\sum_{i <j }(-1)^{i+ j-1}[{ {A_0,...,\hat {A}_i,...,\hat {A}_j,...,A_{2g}}}]d_{i+1}.
\end{equation}
\end{defn}

In summary, we have constructed a chain $c_0+c_1$ that lifts the orbit cycle $o_g$ for the inversive action, i.e. ${q}(c_0+c_1)=[e_g]^v=o_g$. The transition cycle is a boundary in the total complex: $t = \partial(c_0 +c_1).$
Next we  analyze the combinatorics of the transition cycle $t$ and use it to show, in Corollary 6.3 below, that the lifting process terminates.

\begin{subsection}{Combinatorial structure of the transition cycle}

Recall the $\Lambda_{p}$ are the vertical groupoids of the bi-simplicial set constructed from the inversive action of $\Gamma_g^1$ on oriented simplices of inversives. In particular, in the dimension of interest, an object of $\Lambda_{2g-2}$ is an orientation class of a $(2g-1)$-tuple of inversives and a morphism is a mapping class in $\Gamma_g^1$ acting ``inversively" on objects. The vertical homology of the double complex $\mathcal{C}$ in dimension $2g-2$ is precisely the homology of $\Lambda_{2g-2}.$

We decompose the transition cycle $t$ into an alternating  sum of $g+1$ distinct vertical $1$-cycles, $L_0,...,L_{g}$ in the groupoid $\Lambda_{2g-2}.$ 
\begin{equation}\label{transitionSum}
    t = L_0 - L_1 + L_2 + \cdots + (-1)^g L_g  
\end{equation}

The chains $L_0$ and $L_g$ will each be a sum of $2g+1$ terms of (\ref{transitionCycle}). The remaining chains will each be a sum of $2(2g+1)$ terms. Together they have $2g(2g+1)$ terms which agrees with the number of terms in (\ref{transitionCycle}).
In the formulas below integers are taken mod $2g+1.$ 

\vskip -.15in$$L_0 = \sum_{i=0}^{2g} [A_0,...,\hat {A}_i, \hat {A}_{i+1},...,A_{2g}]d_{i+1} \hskip .25in L_g = \sum_{i=0}^{2g}[A_0,...,\hat {A}_i,..., \hat {A}_{i+g},...,A_{2g}]d_{i+g +1}$$
\vskip -.25in 
$$L_k = \sum_{i=0}^{2g} [A_0,...,\hat {A}_i,..., \hat {A}_{i+k},...,A_{2g}]d_{i+k +1} + [0,...,\hat {A}_i,..., \hat {A}_{i+k +1},...,A_{2g}]d_{i+1},                                                                                                                                                                                                                                                                                                                                                                                                            \hskip .075in  1 \leq k \leq g-1.$$

We show in Proposition\ref{VertHom} that the vertical homology class of the transition cycle $t$ is the same as $4g(2g+1)$ times the vertical homology class of $[A_0,A_1,...,A_{2g-2}]\cdot d_{2g}.$ For each $0 \leq k \leq g$, we verify first that the chain $L_k$ is a vertical 1-cycle.  


\begin{prop} For each $0 \leq k \leq g$ the chain $L_k,$  is a vertical $1$-cycle. 
\end{prop}
\begin {proof}
First consider any of the $2g+1$ terms of $L_0.$ Each is on its own a vertical cycle.

Now consider $L_k$ for  $1 \leq k \leq g-1,$ 
and the simplex  $[\hat {A}_0,\ldots, \hat {A}_{k},\ldots,A_{2g}]$ which is the coefficient of the first term in the sum.
To prove the Proposition we construct, for each $k,$  a sequence of composable characteristic morphisms, starting and finishing at the
object $[\hat {A}_0,..., \hat {A}_{k},...,A_{2g}].$  By construction these are $1$-cycles in the groupoid $\Lambda_{2g-2}$,  
hence vertical cycles. Then we show that each sequence represents the $1$-chain $L_k.$ In the following formula, for simplicity, we suppress the coefficients of each characteristic morphism, for they are determined by the initial object. The composite 
\begin{equation}\label{composite}(d_{0}d_{k} ) (d_{2g}d_{k-1}) \cdots (d_{2g-k+2}d_{1})(d_{2g-k+1}d_{0})(d_{2g-k}d_{2g} ) \cdots (d_2d_{k+2})(d_1d_{k+1})
\end{equation} maps the simplex 
$[\hat {A}_0,\ldots, \hat{A}_{k},...,A_{2g}]$ component-wise to 
$[ A_{2g-1},A_{2g},  \hat{A}_0,...,\hat {A}_k,...,A_{2g-2}],$
which is in the same orientation class, as we now observe.
The pair of characteristic homeomorphisms in each set of parentheses corresponds to the composite of two characteristic morphisms with a fixed index $i$
the first one of which comes from the sum on the left and the second from the sum on the right.
The first pair of characteristic morphisms  corresponds to $i=0$ and maps 
$[{ \hat {A}_0,...,\hat {A}_k,...,A_{2g}}]$ 
to $[{A_0, \hat {A}_1,...,\hat {A}_{k+1},...,A_{2g}}].$ Continuing in this manner, we see that the $2(2g+1)$ morphisms in (\ref{composite}) match up with the terms in $L_k,$ as claimed.

When $k = g$ the composite
$d_0(d_{g}d_{2g})( d_{g-1}d_{2g-1})\cdots(d_2d_{g+2})(d_1d_{g+1})$
maps $A_0$ to $ A_g$ and $A_g $ back to $A_0.$
This requires $2g+1$ transformations, so in this case the 
$(2g-1)$-tuple 
$[{\hat {A}_0,...,\hat {A}_{g+1},...,A_{2g}}]$
is restored in $2g+1$ steps, and shows that $L_g$ is a vertical cycle. 
\end{proof}

So each $L_i \in {\mathcal{C}}_{2g-2,1}$ is a vertical $1$-cycle and therefore can be represented by a {\it directed loop} in  $\Lambda_{2g-2}$ 
using the sequence of transformations (\ref{composite}).  Below we use the notation $a \sim_v b$  to  denote that two vertical $1$-cycles $a$ and $b$ are homologous in the  groupoid  $\Lambda_{2g-2}$.

\begin{lem}\label{Lemma1} $ L_1 \sim_v - L_2 \sim_v \cdots  \sim_v (-1)^{g} L_{g-1} \sim_v
 2(-1)^{g+1}L_g.$ 
\end{lem}
\begin{proof} Consider neighboring chains $L_k$ and $-L_{k+1},$ $1 \leq k \leq g-2$ each oriented by the direction of the loops representing them.
 The loop $L_k,$  starting at $[{\hat {A}_0},...,{\hat {A}_{k+1}},...,A_{2g}],$ is formed by a sequence of morphisms using the composite 
 $(d_{k+1}d_{0})(d_{k} d_{2g})  \cdots (d_{k+3}d_2)(d_{k+2}d_1). $ If we flip the transformations within each set of parentheses,
$(d_{0}d_{k+1}) \cdots (d_2d_{k+3})(d_1d_{k+2}),$ we obtain the loop $L_{k+1}$ starting at $[\hat {A}_0,...,\hat {A}_{k+1},...,A_{2g}].$
Any square formed by pairs of consecutive morphisms in the two loops has the form

$$d_{m}$$
\vskip -.35in $$ \longrightarrow$$
 \vskip -.35in $$ d_{n} \uparrow  \hskip .5in \uparrow  d_{n} $$
\vskip -.35in $$\longrightarrow$$
\vskip -.35in $$d_{m}$$ where horizontal followed by vertical is in $L_k$ and vertical followed by horizontal is in $L_{k+1}.$
 For each square the separation between vertices indexed by $m$ and those by $n$ is 
at least $1$. Then the characteristic mapping classes $d_{m}$ and $d_{n}$ commute  since the loop  $d_n^{-1}d_m^{-1}d_nd_m$  fixes all the ${\rm a}_i.$
Consequently each square closes up as a vertical chain. 
All loops  $L_{k}$ and $-L_{k+1},$ for $k \neq 0, k \neq g-1,$ are therefore freely homotopic. In the exceptional case recall $L_{g}$ closes up with half the morphisms of $L_{g-1}$ so that $L_{g-1}$ is freely homotopic, with a change in sign, to twice $L_g.$ Free homotopy of a pair of loops implies the loops are homologous. 
\end{proof}

\begin{lem}\label{Lemma2} $L_0 +  L_1+ \cdots + L_g \sim_v 0. $
\end{lem}
 \begin{proof} The $2g(2g+1)$ distinct morphisms which are the terms of the sum  $L_0 +  L_1+ \cdots + L_g$
 are in $1-1$ correspondence with the morphisms forming the composite
$(d_{2g}d_{2g-1} \cdots d_{0})^{4g}$, which is the identity in $\Gamma_g^1$ by (\ref{identities}).
That means there is a composite in the groupoid of all the morphisms whose square is the identity, so as a cycle the sum is homologous to $0.$ 
In fact it is represented by a loop which is trivial homotopically.
\end{proof}

\begin{prop}\label{VertHom}
$ t \sim_v (4g)(2g+1) [A_0,A_1,\ldots,A_{2g-2}]\cdot d_{2g}.$ 
\end{prop}
\begin{proof} 
Combining the results of Lemmas \ref{Lemma1} and \ref{Lemma2} gives
\begin{equation}\label{eq10}
    L_g \sim_v (-1)^g\ L_0.
\end{equation}
Applying Lemma \ref{Lemma2} to the expression (7) for $t$ gives
\begin{equation}\label{eq11}
t  \sim_v L_0 +(-1)^g (2g-1)L_g.
\end{equation}
After eliminating $L_g$ from  (\ref{eq10}) and (\ref{eq11}) we obtain
$t \sim_v 2g\ L_0. $
 Now $L_0$ is a sum of $2g+1$ distinct terms but each is counted twice so there are $2(2g+1)$ isotropic terms in all. Isotropic morphisms  are freely homotopic for they 
are conjugate. In particular each is freely homotopic to $d_{2g}.$ This gives the desired formula for $t.$ 
 \end{proof}

\begin{example} We illustrate the constructions of this section with the genus $g=1$ case. The  mapping class group $\Gamma_1^1$ is $SL(2, \Bbb Z).$ A characteristic simplex, together with its associated characteristic  mapping classes, is given by the data
$${{\rm a}}_0 = (1,0) \hskip .5in {{\rm a}}_1 = (0,1) \hskip .5in {{\rm a}}_2 = (-1,-1)/{\sqrt 2}$$
$$S=\left( 
\begin{array}{ccc}
  0&   -1\\
  1&    0 \\
 
\end{array}
\right) \hskip .25in 
T= \left(
\begin{array}{ccc}
  0&   -1   \\
  1&  -1   \\
 \end{array}
\right)
\hskip .25in  d_0=\left(
\begin{array}{ccc}
  1&    0   \\
  1&   1   \\
 \end{array}
\right) \hskip .25in 
d_1 = \left(
\begin{array}{ccc}
 2&   -1   \\
  1&  0   \\
 \end{array}
\right)
\hskip .25in 
d_2= \left(
\begin{array}{ccc}
  1&   -1  \\
  0&   1  \\
 \end{array}
\right).$$

\begin{figure}[htbp] 
   \centering 
   \includegraphics[width=2.75in]{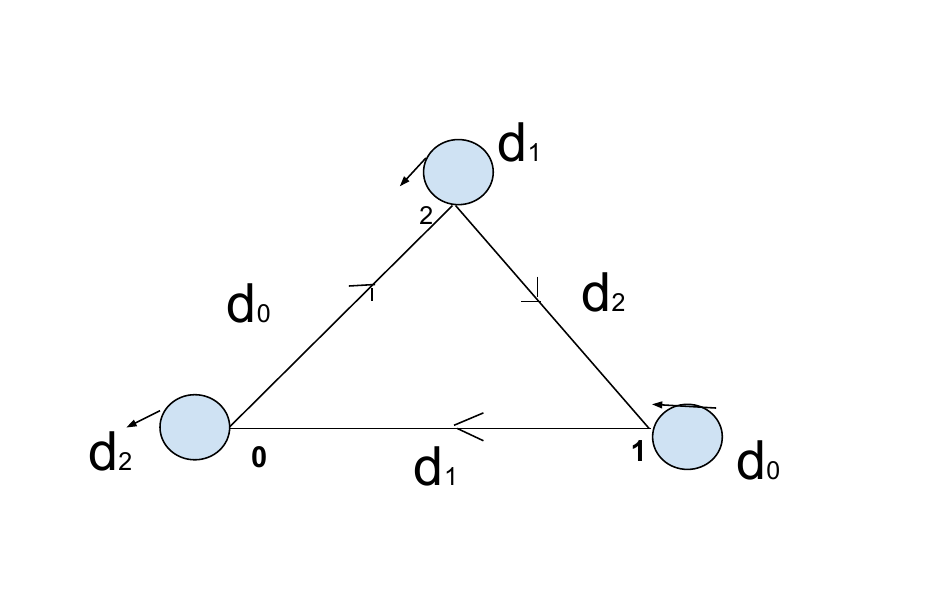} 
   \caption{\small The transition cycle $t$ for genus $g=1$}
   \label{tcycle}
\end{figure}

 We consider $SL(2, \Bbb Z)$ acting on inversives. Notice that in this case inversives are in one to one correspondence with lines in the plane that pass through the origin, and the action of $SL(2, \Bbb Z)$ factors through $PSL(2, \Bbb Z)$. 

The chain $c_1$ is 
$[{ A_1,A_2}]d_1 - [{ A_0,A_2}]d_2 + [{ A_0,A_1}]d_0.$

The transition cycle $t$ is
$([{ A_2}] - [{ A_1}])d_1 + ([{ A_0}] - [{A_2}])d_2 + ([{A_1}] - [{A_0}])d_0.$

The terms with a positive coefficient correspond to characteristic morphisms which fix an object. 
The three with a negative coefficient move  ${ A_1}$ to ${A_0},$ ${ A_2}$ to ${A_1},$ and ${ A_0}$ to ${ A_2}.$
The transition cycle is illustrated in figure \ref{tcycle}, with $A_i$ denoted by $i$, and where the loops stand for the characteristic morphisms which fix objects,
and arrows for the characteristic morphisms that move objects.
We can write the transition cycle $t$ as  $t = L_0  - L_1$ where 
$L_0 = [{ A_2}]d_1 + [{ A_0}]d_2  + [{A_1}]d_0$ and  $L_1 = [{ A_1}]d_1 + [{A_2}]d_2 + [{A_0}]d_0$.

The characteristic morphisms have explicit formulas in $SL(2, \Bbb Z).$ We can therefore compute directly the element in the isotropy group of $[A_0]$ determined by $t$.
$$(d_0^{-1}d_1 d_2^{-1} d_0 d_1^{-1} d_2)^2=\left(
\begin{array}{ccc}
  -1&    6   \\ 
   0&   -1   \\ 
 \end{array}\right)^2 = -12 d_2.$$

\end{example}

\end{subsection}

\section {Torsion and the order of $\E^g$}\label{Torsion}

In this section we show that $\E^g$ is a torsion class, and prove Theorem A.

\subsection{Holonomy of the transition cycle}\label{SecHolonomy}

Generally, consider a connected discrete groupoid $\Lambda$ and the isotropy group $\Lambda_{\rm b}$ of a base point ${\rm b}.$
For each object ${\rm y}$ of $\Lambda$ choose a morphism, $m_{\rm y}$  from ${\rm b}$ to ${\rm y}.$  The  {\it holonomy homomorphism} from groupoids to groups
$\Lambda  \to \Lambda_{\rm b}$,  defined by $f \mapsto  m_{t(f)}^{-1}\circ f  \circ m_{s(f)}$, 
is a natural transformation of categories. It determines  a homotopy equivalence on realizations 
$B\Lambda \to B\Lambda_{\rm b}$; 
as a result $B\Lambda$ is a $K(\Lambda_{\rm b},1).$ Moreover, the holonomy homomorphism induces an isomorphism on homology which is independent of the choice of morphisms $m_{\rm y}.$

 Let $\Lambda$ be the component of $\Lambda_{2g-2}$ containing the object ${[{A}]}:=[A_0,A_1,...,A_{2g-2}]$ and 
denote by $\Lambda_{[A]}$ the isotropy group of that object. The  holonomy homomorphism induces an isomorphism $H_1(\Lambda)\rightarrow H_1(\Lambda_{[A]}).$   

Notice that $\Lambda_{[A]}$ is a subgroup of $\Gamma_g^1$ that can also be identified as a subgroupoid of $\Lambda.$ The transition cycle $t$ is a $1$-cycle in $\Lambda$ and $[A_0,A_1,\ldots,A_{2g-2}]\cdot d_{2g}$ is a $1$-cycle in $\Lambda_{[A]} \subset \Lambda.$ By Proposition \ref{VertHom} the cycles $t$ and $2g(2g+1)[A_0,A_1,\ldots,A_{2g-2}]\cdot d_{2g}$ are homologous in $\Lambda.$ 

\begin{prop}\label{infiniteOrder} The characteristic mapping class $d_{2g}$  represents an element of infinite order in $H_1( \Lambda_{[A]}).$
\end{prop}
\begin{proof} 
Consider $H_1(\Sigma_g) \cong \Bbb Z^{2g}$ generated by ${\rm a}_0,\ldots,{\rm a}_{2g-1}$.  The action of the mapping class group $\Gamma_g^1$ on $H_1(\Sigma_g)$ gives, with respect to this generating set, a group homomorphism  $\psi: \Gamma_g^1\rightarrow SL(2g, \Bbb Z)$
defined by $f\mapsto f_*$. 

If  $f\in\Lambda_{[A]}$, then $f$ is a mapping class that permutes the inversives $A_0,A_1,\ldots, A_{2g-2}$ preserving the orientation class $[A_0,A_1,\ldots, A_{2g-2}]$. Moreover, the mapping class $f$ must preserve the cyclic ordering (\ref{order1}).   Then $f_*$ acts on the generators ${\rm a}_0,\ldots,{\rm a}_{2g-2}$ by $\pm P$, where $P$ is a $(2g-1)\times(2g-1)$ permutation matrix with $\det P=1$. 
The image $K= \psi(f)$ is a matrix of the form  $K = QT$ below,
where $I$ is the identity $(2g-1)\times(2g-1)$ matrix and $n_i\in\mathbb{Z}$ for $0\leq i\leq 2g-2$.

$$K =  \pm\begin{pmatrix}
&  &  {\cdots}& & n_0 \cr
&  &  {\cdots} & & n_1 \cr
{} &{} & P  & {}& \vdots\cr
&  &{\cdots}  & & n_{2g-2} \cr
0& & \cdots & 0&   1\cr
\end{pmatrix}
\hskip .025in
=
\hskip .025in
\begin{pmatrix}
&  &  {\cdots}& &  n_0 \cr
&  &  {\cdots} & &  n_1 \cr
{} &{} & I  & {}& \vdots\cr
&  &{\cdots}   & &  n_{2g-2} \cr
0& & \cdots & 0&  1\cr
\end{pmatrix}
\begin{pmatrix}
&  &  {\cdots}& & 0 \cr
&  &  {\cdots} & & 0 \cr
{} &{} & \pm P  & {}& \vdots\cr
&  &{\cdots}  & & 0 \cr
0& & \cdots & 0&  \pm 1\cr
\end{pmatrix}
$$

Matrices of the form $K$ and $Q$ are closed under products and inverses so they determine subgroups,  ${\mathcal{K}}$  and ${\mathcal{Q}},$ of  $SL(2g,\mathbb{Z}).$ The subset ${\mathcal{T}}$ of  ${\mathcal{K}}$ which consists of matrices of the form $K$ with $n_0 + \cdots +n_{2g-2} = 0$ is a normal subgroup of ${\mathcal{K}}$.

We claim that the quotient group ${\mathcal{K}}/{\mathcal{T}}$ is abelian.  Indeed, consider the composite ${\mathcal{Q}} \hookrightarrow {\mathcal{K}} \twoheadrightarrow {\mathcal{K}}/{\mathcal{T}}. $ The image of this homomorphism is the subgroup of cosets represented by $Q \in {\mathcal{Q}}.$
The product decomposition $K=QT$ shows that every coset in ${\mathcal{K}}/{\mathcal{T}}$ has a representative in ${\mathcal{Q}}$. Therefore,  the  homomorphism ${\mathcal{Q}} \rightarrow {\mathcal{K}}/{\mathcal{T}}$ is onto. It follows that ${\mathcal{K}}/{\mathcal{T}}$ is isomorphic to a quotient of ${\mathcal{Q}}$ which must be abelian since ${\mathcal{Q}}$ is abelian.

Hence, the  composite 
$\Lambda_{[A]}\xrightarrow{\psi} {\mathcal{K}} \twoheadrightarrow {\mathcal{K}}/{\mathcal{T}}$
factors through the abelianization $H_1(\Lambda_{[A]})$. Moreover, the  mapping class $d_{2g}^{-1}$ maps under this composite to an element in ${\mathcal{K}}/{\mathcal{T}}$ of infinite order, namely the coset represented by the matrix $Q$  in ${\mathcal{Q}}$ with  $n_i =1$ for $i=0,\ldots, 2g-2$. So the characteristic mapping class $d_{2g}$ has infinite order in $H_1( \Lambda_{[A]}).$ 
\end{proof}

\begin{cor}\label{Infinite}
The transition cycle $t$ is a $1$-cycle in the groupoid $\Lambda$ that represents a class of infinite order in $H_1( \Lambda).$ 
\end{cor}
\begin{proof}
 By Proposition \ref{infiniteOrder}, the characteristic mapping class $d_{2g}$ represents a class of infinite order in $H_1( \Lambda_{[A]})$, then so does $4g(2g+1)d_{2g}.$
 From Proposition \ref{VertHom} the transition cycle $t$ and the $1$-cycle $4g(2g+1)[A_0,A_1,\ldots,A_{2g-2}]\cdot d_{2g}$ represent the same class in $H_1( \Lambda).$    Since the holonomy homomorphism $\Lambda\rightarrow \Lambda_{[A]}$ induces an isomophims in homology, the corollary follows.
\end{proof}

 We are now prepared to show that the transition cycle $t$ gives an obstruction to the existence of a cycle in the total complex that lifts $o_g$.

\begin{cor}\label{NOLIFT}
  Neither the orbit chain $(o_g/2)$, nor any multiple of it, lifts to a cycle in the total complex of the inversive action. That is,  there does not exist a cycle $c$ in the total complex such that ${\rm q}(c)=k(o_g/2)$, for any $k\in\mathbb{Z}-\{0\}$. 
\end{cor}
\begin{proof} We attempt to lift the orbit chain $k(o_g/2)$ to a cycle in the total complex, as explained in the introduction of Section \ref{SecTranCycle}, starting with the chain $c_0=k(e_g/2)\in \mathcal{C}_{2g,0}$. 

If $k$ is odd, then there is no chain $c_1\in {\mathcal{C}}_{2g-1,1}$ such that ${\rm q}(c_0+c_1)=k(o_g/2)$, for its chain boundary is a sum of an odd number of distinct faces so orbits cannot cancel. 

If $k=2m$ is even, we can continue the lifting process of by $k(o_g/2)$ taking 
$$c_1=  m\sum_{ i=0}^{2g}(-1)^{i}[ {A_0,...,\hat {A}_{i}, A_{i+1},...,A_{2g}}]d_{i+1}\in {\mathcal{C}}_{2g-1,1},$$  where $i \equiv n {\rm \hskip .025in mod}(2g+1)$. 
If a further extension were to exist, there would be a chain $c_2$ satisfying $-\partial^v c_2 =  \partial^h c_1 = mt.$ Since $c_2$ is a $2$-chain in the groupoid $\Lambda_{2g-2}$, this implies that a non-trivial multiple $t$ bounds vertically in $\Lambda_{2g-2}.$ However, by Corollary \ref{Infinite}, the transition cycle $t$ has infinite order  in $H_1(\Lambda)$, and hence in $H_1(\Lambda_{2g-2})$. Therefore, no such chain $c_2$ can exist, the lifting terminates, and there is no cycle $c$ such that ${\rm q}(c)=k(o_g/2)$. 

\end{proof}

\begin{subsection}{The cohomology class $\E^g$ in $Ext$}\label{SecTorsion}

We use our approach to directly prove that $\E^g$ is a torsion class. This follows from results on the rational cohomology of moduli spaces obtained by algebro-geometric techniques \cite{Io02, Lo95}, but Theorem \ref{Theo:TorsionClass} provides a self-contained, intrinsic proof. 

Consider the homomorphism $\beta: H^{2g}(T{\mathcal{C}}; \Bbb Z)\rightarrow Hom(H_{2g}(T{\mathcal{C}}), \Bbb Z)$ of the Universal Coefficient Theorem (\ref{UCT}).

\begin{theorem}\label{Theo:TorsionClass} The homomorphism $\beta$ maps $\E^g$ to $0,$ hence $\E^g$ is a  torsion class.
\end{theorem}
\begin{proof} By Proposition \ref{InvCocycle} the class $2\E^g$ is represented by the cocycle  $(e_g/2)^*$ satisfying $\langle (e_g/2)^*, (e_g/2) \rangle = 1.$
Since $(e_g/2)^*$ is a homomorphism from $T\mathcal{C}_{2g}$ onto $\Bbb Z$, 
it splits the $2g$-chains $T\mathcal{C}_{2g}$ into a summand $C$ generated by the chain $(e_g/2)$ and a summand $C'$ on which $(e_g/2)^*$ is trivial. The homomorphism $\beta(2\E^g)\in Hom(H_{2g}(T{\mathcal{C}}), \Bbb Z)$ is represented by the restriction of the cocycle $(e_g/2)^*$ to $2g$-cycles $\mathcal{K}_{2g}$ in the total complex $T\mathcal{C}$. 

If $\beta(2\E^g)$ were non-trivial in $Hom(H_{2p}(T{\mathcal{C}}), \Bbb Z)$ there would exist a $2g$-cycle $z$ with the property that $(e_g/2)^*(z)\neq 0$. Therefore,  when the cycle $z$ is written in terms of a basis determined by the splitting of $T\mathcal{C}_{2g}$ as $C\oplus C'$, it must have a non-trivial term in $C.$ However, by Corollary \ref{NOLIFT}, neither the chain $(e_g/2)$ nor any non-zero multiple of it, can be extended, by adding terms in $C',$ to a cycle in the total complex.  Consequently $z$ must be an element of the subgroup $C' $ in which case we obtain $(e_g/2)^*$ is the zero homomorphism when restricted to $\mathcal{K}_{2g}$. We conclude that $\beta$ maps $2\E^g$ to $0$, which implies that $\beta$ maps $\E^g$  to $0$ as well.
\end{proof}

Now let us recall the definition of 
$\alpha: Ext(H_n( T{\mathcal{C}}), \Bbb Z)\rightarrow H^{n+1}( T{\mathcal{C}};\mathbb{Z})$
in the Universal Coefficient Theorem (\ref{UCT}). Let ${\mathcal{K}}_{n}$ denote the $n$-cycles and ${\mathcal{B}}_{n}$ the $n$-boundaries of the total complex $T\mathcal{C}$.
The group  $Ext(H_n( T{\mathcal{C}}), \Bbb Z)$ is isomorphic to  $Hom({\mathcal{B}}_{n}, \Bbb Z)/\iota(Hom({\mathcal{K}}_{n}, \Bbb Z))$, where $\iota$ is the homomorphism which restricts an element of $Hom({\mathcal{K}}_{n}, \Bbb Z)$ to ${\mathcal{B}}_{n}.$  Given $f \in Hom({\mathcal{B}}_{n}, \Bbb Z)$ representing a class in $Ext(H_{n}(T{\mathcal{C}}), \Bbb Z),$  the element $\alpha(f) \in 
Hom(( T{\mathcal{C}})_{n+1}, \Bbb Z)$ representing a class in $H^{n+1}( T{\mathcal{C}}; \Bbb Z)$ is defined by $$\langle \alpha(f), c \rangle = \langle f, \partial c \rangle\text{,\ \  for any\  }c\in (T{\mathcal{C}})_{n+1}.$$

Consider the chains $c_0=e_g$, 
 $c_1$ and the transition cycle $t=\partial(c_0+c_1)$ as in Section \ref{SecTrans}, and let $(b/2):=\partial(e_g/2)$.
 
\begin{prop}\label{ExtRep-1} There  is  a well-defined homomorphism $\chi \in Hom({\mathcal{B}}_{2g-1}, \Bbb Z)$ which represents $\E^g$ in
 $Ext(H_{2g-1}(T{\mathcal{C})}, \Bbb Z)$ and satisfies $\langle \chi, (b/2) \rangle = 1.$  
 \end{prop} 
 \begin{proof} By definition ${\rm q}(e_g/2)=[(e_g/2)]^v=(o_g/2)$. Then $$\langle (e_g/2)^*, (e_g/2)  \rangle=\langle (o_g/2)^*, {\rm q}(e_g/2 ) \rangle=\langle (o_g/2)^*, (o_g/2) \rangle=1.$$
By Theorem \ref{Theo:TorsionClass} the class $ \E^g$, represented by the cocycle $(e_g/2)^*$, is a torsion class. Hence, the Universal Coefficient Theorem (\ref{UCT}) implies that there exists $\chi \in Hom({\mathcal{B}}_{2g-1}, \Bbb Z)$ such that $\alpha(\chi) = (e_g/2)^*.$ Moreover, it satisfies
 \vskip .15in
 \hskip .75in $\langle \chi, (b/2) \rangle=\langle \chi, \partial(e_g/2 ) \rangle=\langle \alpha(\chi), (e_g/2)  \rangle=\langle (e_g/2)^*, (e_g/2)  \rangle=1.$  \end{proof} 

\begin{prop}\label{ExtRep} The homomorphism $\chi $ satisfies $\langle \chi, t \rangle = 2.$  
 \end{prop}  
 \begin{proof} We have ${\rm q}(c_0)={\rm q}(c_0+c_1)=[e_g]^v=o_g$, and $ t=\partial(c_0+c_1)$. Therefore, $$\langle (e_g/2)^*, c_0 +c_1 \rangle=\langle (o_g/2)^*, {\rm q}(c_0 +c_1) \rangle=\langle (o_g/2)^*, o_g \rangle=2.$$
By Theorem \ref{Theo:TorsionClass} the cohomology class $2\E^g$, represented by the cocycle $(e_g/2)^*$, is a torsion class. Hence, the Universal Coefficient Theorem (\ref{UCT}) implies that there exists $\chi \in Hom({\mathcal{B}}_{2g-1}, \mathbb{Z})$ such that $\alpha(\chi) = (e_g/2)^*.$ Moreover, it satisfies

\vskip .15in
\hskip .75in  $\langle \chi,  t \rangle=\langle \chi, \partial(c_0 + c_1) \rangle=\langle \alpha(\chi), c_0 + c_1 \rangle=\langle (e_g/2)^*, c_0 +c_1 \rangle=2.$ \end{proof} 
 
\end{subsection}

\begin{subsection}{Detecting the order of torsion at the threshold}\label{DETECT}
We showed in {Theorem B} that $\E^g$ is a torsion  class in $H^{2g}(\Gamma_g^1;\mathbb{Z})$. Furthermore, $\E^g$ is known to be non-trivial since ${\Gamma}_g^ 1$ has non-trivial finite cyclic subgroups; see for instance \cite[Theorem A]{JJR20}.  We use the setting developed in this paper to independently show non-triviality of $\E^g$ and derive information about its order.

\begin{rem}\label{NONTRIV}
Let $m\geq 0$ be the order of $2\E^g$. By Proposition 6.4  the cohomology class $2\E^g$ can be represented in $Ext(H_{2g-1}(T{\mathcal{C})}, \Bbb Z)$ by  $\chi \in Hom({\mathcal{B}}_{2g-1}, \mathbb{Z)}.$ 
 Therefore, the multiple $m\chi$ can be extended to an element of $Hom({\mathcal{K}}_{2g-1}, \mathbb{Z}).$
Note, the multiple $m$ cannot be $0$ since $\chi$ has a  non-trivial evaluation, namely $\langle \chi, (b/2) \rangle = 1$ (see Proposition \ref{ExtRep}), so we must have that $m>0$.  Hence,  $2\E^g$, and therefore $\E^g$, are non-trivial torsion classes. 
\end{rem}

Suppose there exists ${\rm d}\in \Lambda_{[A]}\subset \Lambda$ and $\lambda>0$ such that $t\sim_{v} \lambda {\rm d}$. Then  $\lambda{\rm d} -  t = \partial^v c_2$ for some chain $c_2 \in {\mathcal{C}}_{2g-2,2}.$ It follows that $\partial(c_0 + c_1 + c_2) = \lambda{\rm d} +R,$ where $R := \partial^h c_2 \in {\mathcal{C}}_{2g-3,2}$.

 \begin{prop}\label{ExtRep2} The homomorphism $\chi \in Hom({\mathcal{B}}_{2g-1}, \mathbb{Z})$
 satisfies $\langle \chi, \lambda {\rm d} +R \rangle = 2.$  
 \end{prop} 
 \begin{proof}
   First observe that ${\rm q}(c_0 +c_1+c_2)=[c_0]^v=o_g$, then $$\langle (e_g/2)^*, c_0 +c_1+c_2 \rangle=\langle (o_g/2)^*, {\rm q}(c_0 +c_1+c_2) \rangle=\langle (o_g/2)^*,o_g \rangle=2.$$
   Since $\alpha(\chi)=(e_g/2)^*$, it follows that\bigskip
   
 \ \ \   $2 = \langle (e_g/2)^*, c_0 +c_1+c_2 \rangle =  \langle \alpha(\chi), c_0 + c_1 +c_2\rangle = \langle \chi, \partial(c_0 + c_1+c_2 ) \rangle = \langle \chi, \lambda {\rm d} +R \rangle.$\ \ \end{proof}

\begin{prop}\label{ExtRep3} The homomorphism $m \chi$  extends to $Hom({(T\mathcal{C}})_{2g-1}, \mathbb{Z})$ and the extension 
satisfies $2m=\langle m \chi, \lambda {\rm d}  \rangle =\lambda \langle m \chi, {\rm d}  \rangle$. The integer $\lambda$ must then divide $m$.
 \end{prop} 

\begin{proof}  
 By Proposition 6.6  the 
evaluation of $m\chi$ on the chain $\lambda {\rm d} +R$ can be computed as 
$$\langle m\chi,\lambda {\rm d}  \rangle + \langle m\chi, R \rangle=\langle m\chi, \lambda {\rm d}  +R\rangle =m\langle \chi, \lambda {\rm d}  +R\rangle =2m.$$
To complete the proof we show that $m\chi$ evaluates $R$ to $0,$ hence
$\lambda d$ to $2m.$ 

Consider the differential $\partial = \partial^h +\partial^v$ of the total complex $T{\mathcal{C}}$, restricted to ${\mathcal{C}}_{2g-2,2}$,  
 and the following short exact sequences: 
 
 $$0 \longleftarrow H \buildrel {\partial^h} \over {\longleftarrow} {\mathcal{C}}_{2g-2,2} \longleftarrow ker {\partial^h}  \longleftarrow 0, \hskip .1in H \subset 
 {\mathcal{C}}_{2g-3,2}$$ 
  $$0 \longleftarrow V \buildrel {\partial^v} \over {\longleftarrow} {\mathcal{C}}_{2g-2,2} \longleftarrow ker {\partial^v}  \longleftarrow 0,
   \hskip .1in V \subset 
 {\mathcal{C}}_{2g-2,1}$$ 
   $$0 \longleftarrow H \oplus V \buildrel {\partial^h + \partial^v } \over {\longleftarrow} {\mathcal{C}}_{2g-2,2} \longleftarrow ker ({\partial^h +\partial^v})  \longleftarrow 0.$$

The groups $H$ and $V$ are free so the homomorphisms $\partial^h,$ $\partial^v$  and $\partial=\partial^h +\partial^v$ split. Let
$\sigma$ denote  the splitting homomorphism for $\partial$.
We claim $\sigma( H \oplus V) =   \sigma{\vert_ {H}(H)} \oplus \sigma{\vert_ {V}(V)}.$
Indeed, since $\sigma$ is $1-1$ the image of $\sigma$ is a subgroup isomorphic to  $H \oplus V.$ An element $\alpha$ in the image of $\sigma$  has the form
$\alpha = \sigma(h+v)$ for unique elements $h \in H,$ $v \in V.$ Then also $\alpha = \sigma(h)+\sigma(v)$ for unique elements $h \in H,$ $v \in V,$ which verifies the claim.

Denote $\sigma{\vert_ {H}}$ and $\sigma{\vert_ {V}}$ by $\sigma^h$ and $\sigma^v.$ Note that $\sigma^h$ splits $\partial^h$ and $\sigma^v$ splits $\partial^v.$ 
Note also that $\sigma^h(H) = \sigma( H, 0)$ and $\sigma^v(V) = \sigma( 0, V)$ have only the $0$-element in common.

Now ${\mathcal{C}}_{2g-2,2} = \sigma^h(H) \oplus ker \partial^h = \sigma^v(V) \oplus ker \partial^v,$ hence $\sigma^h(H) \subset ker \partial^v.$
Then $\partial^v \sigma^h =0,$ and in particular $\partial^v \sigma^h(R)=0.$ By definition  of the splitting $\partial^h \sigma^h(R)=R$ so the chain $\sigma^h(R)$ bounds $R$,  i.e. $\partial(\sigma^h(R))=R$.  
Therefore,
 $$\langle m\chi,  R \rangle=\langle m\chi, \partial(\sigma^h(R) )\rangle=\langle \alpha(m\chi), \sigma^h(R) \rangle =\langle m(e_g/2)^*, \sigma^h(R)  \rangle=0.$$
The evaluation is 0 since $\sigma^h(R) \in {\mathcal{C}}_{2g-2,2}$ is in the subgroup complementary to the summand generated by $(e_g/2).$
\end{proof}

\end{subsection}

\begin{subsection}{Proof of Theorem \ref{MAIN}}\label{MAINProof}

We first use the approach from Section \ref{DETECT} to prove that $\E^g$  has order a positive multiple of  $4g(2g+1)$. This follows from the existence of finite cyclic subgroups of order $4g$ and $2g+1$ in ${\Gamma}_g^ 1,$ (see \cite[Theorem A]{JJR20}).  Furthermore, our approach  will allow us  to detect torsion of order $2g-1$ in our proof of Theorem \ref{MAIN}, which,  as opposed to torsion of order $4g$ and $2g+1$,  is not obtained using periodic elements in the  mapping class group.

\begin{prop} \label{NonTrivialAgain} 
$\E^g$ is a torsion class, and its order is a positive multiple of $4g(2g+1).$
\end{prop}
\begin{proof}  Let $m>0$ be the order of $2 \E^g$, then $\E^g$ has order $2m$.  By Proposition \ref{VertHom}, $t\sim_{v} \lambda {\rm d}$ where $\lambda = 4g(2g+1)$ and  ${\rm d}=[A_0,...,A_{2g-2}]\cdot d_{2g}\in\Lambda_{[A]}\subset\Lambda$. Proposition \ref{ExtRep3} implies that $\lambda=4g(2g+1)$ must divide $2m$ as claimed.
\end{proof}

To find  torsion of order $2g-1$ in cohomology, we consider a $w\in Aut(\pi_1(\Sigma_g, z))$ defined on the generators $\{{{{\rm a}}}_0,...,{{{\rm a}}}_{2g-2},{{{\rm a}}}_{2g-1}\}$ of $\pi_1(\Sigma_g, z)$ as follows:
$${\rm a}_0 \to {\rm a}_1 \to {\rm a}_2\to \cdots \to {\rm a}_{2g-2} \to {\rm a_0} \hskip .5in {{\rm a}}_{2g-1 }  \to {{\rm a}}_0^{-1} {{\rm a}}_{2g-1} {{\rm a}}_0^{-1}. $$  
It can be seen that $w$ is a well-defined element of $ Aut(\pi_1(\Sigma_g, z))$ by checking that the defining relation (\ref{fundamentalRelation}) is preserved.

\begin{prop}\label{w}
The automorphism $w$ is a mapping class in $\Gamma_g^1$ such that $w^{2g-1}=d_{2g}^2$.
\end{prop}
\begin{proof}
The automorphism $w^{2g-1}$ maps 
${{\rm a}}_{2g-1}$ to $({{\rm a}}_{2g-2}^{-1} \cdots {{\rm a}}_0^{-1}) {{\rm a}}_{2g-1} ({{\rm a}}_0^{-1} \cdots {{\rm a}}_{2g-2}^{-1})$
and fixes all the remaining ${\rm a}$'s.
On the other hand, recall $d_{2g}^{-1}$ maps ${\rm a}_{2g-1}$ to ${\rm a}_{2g}^{-1}={{\rm a}}_0 \cdots {{\rm a}}_{2g-1} = 
{{\rm a}}_{2g-1} \cdots {{\rm a}}_{0}$ (see Section \ref{CHAR}).
Composing  $w^{2g-1}$ with $d_{2g}^{-1}$ on the left  and using the defining relation gives
$ {{\rm a}}_{2g-1} \to  {{\rm a}}_{2g-1}({{\rm a}}_0^{-1} \cdots {{\rm a}}_{2g-2}^{-1})$
and composing again gives
$d_{2g}^{-2}w^{2g-1}: {{\rm a}}_{2g-1} \to  {{\rm a}}_{2g-1}$ so that $d_{2g}^{-2}w^{2g-1}$
is the identity automorphism. Note $w$ is orientation preserving since $w^{2g-1}=d_{2g}^2$ is and $2g-1$ is odd.   
\end{proof}

\begin{cor}\label{tw}  $ t \sim_v (2g)(2g+1)(2g-1) [A_0,\ldots,A_{2g-2}]\cdot w.$ 
\end{cor}
\begin{proof}  Since $w$ acts by an even cyclic permutation on the generators $\{{{{\rm a}}}_0,...,{{{\rm a}}}_{2g-2}\}$, it fixes $[A]=[A_0,\ldots, A_{2g-2}]$. Then $w\in \Lambda_{[A]}$ and  Proposition \ref{w} implies that  $2d_{2g} = (2g-1)w$ holds in $H_1( \Lambda_{[A]})$. Hence, from the isomorphism in homology induced by the holonomy homomorphim $\Lambda\rightarrow \Lambda_{[A]}$, we have that
$2[A_0,\ldots,A_{2g-2}]\cdot d_{2g} \sim_v (2g-1)[A_0,\ldots,A_{2g-2}]\cdot w$. 
The corollary follows then from Proposition \ref{VertHom}.
\end{proof}

\noindent {\bf Theorem A.} The order of $\E^g$ is a positive integer multiple of  $4g(2g-1)(2g+1).$
\begin{proof}  

Let $m>0$ be the order of $2\E^g$, then  $2m$ is the order of $\E^g$.  
From Corollary \ref{tw}, we have that  $t\sim_{v} \lambda {\rm d}$ where $\lambda = (2g)(2g+1)(2g-1) $ and  ${\rm d}=[A_0,\ldots,A_{2g-2}]\cdot w\in \Lambda_{[A]}\subset\Lambda$. Proposition  \ref{ExtRep3} implies $\lambda=(2g)(2g+1)(2g-1)$ must divide $2m$.  From Propostion \ref{NonTrivialAgain} we have that $4g$ also divides $2m$. Since $(2g+1)(2g-1)$ and $4g$ are relatively prime, the statement of the theorem follows.
\end{proof}

\end{subsection}

\subsection*{Acknowledgements}
We are grateful to Kathryn Mann, Israel Morales and Sam Nariman for useful communications. This paper was partially written while the second author was visiting Northeastern University with funding from the National University of Mexico through a DGAPA-UNAM PASPA sabbatical fellowship.  She was also funded by DGAPA-UNAM grant PAPIIT IA104010 when this project started. She is grateful to DGAPA-UNAM and thanks the NEU Department of Mathematics and the first author for their hospitality.


\bibliographystyle{alpha}

\bigskip


\begin{small}

\noindent Mathematics Department, Northeastern University, Boston MA,  USA 02115.\\   E-mail: \texttt{
{s.jekel@northeastern.edu}}\medskip

\noindent Instituto de Matem\'aticas, Universidad Nacional Aut\'onoma de M\'exico. Oaxaca de Ju\'arez, Oaxaca, M\'exico 68000.  E-mail: \texttt{
{rita@im.unam.mx}}
\end{small}

\end{document}